\documentclass[a4paper,leqno]{amsart}

\usepackage[alphabetic]{amsrefs}
\usepackage{enumerate}

\usepackage{amsmath}

\usepackage{mathdots} 

\usepackage{amssymb}

\usepackage{amsbsy}

\usepackage{enumerate}

\usepackage[alphabetic]{amsrefs}

\usepackage{mathptmx} 

\usepackage{hyperref}

\usepackage{mathrsfs} 

\usepackage{txfonts}  

\usepackage{stmaryrd} 

\usepackage{fancyhdr}  

\usepackage{comment}

\usepackage{algorithm}

\usepackage{algpseudocode}

\usepackage{tikz}

\usetikzlibrary{arrows,positioning}

\usepackage{makecell}

\usepackage{pgfplots}

\usepackage{graphicx,amssymb}

\usepackage{ifthen}

\usepackage{epsfig}

\newtheorem{theo}{Theorem}

\newtheorem{lemm}{Lemma}

\newtheorem{coro}{Corollary}

\newtheorem{defi}{Definition}

\newtheorem{rema}{Remark}

\newtheorem{prob}{Problem}

\renewcommand{\le}{\leqslant}

\renewcommand{\ge}{\geqslant}

\newcommand{\vep}{\varepsilon}

\newcommand{\ds}{\displaystyle}

\newcommand{\bH}{\ensuremath{\mathbb H}}

\newcommand{\bN}{\ensuremath{\mathbb N}}

\newcommand{\bR}{\ensuremath{\mathbb R}}

\newcommand{\bZ}{\ensuremath{\mathbb Z}}

\newcommand{\sD}{\ensuremath{\mathsf D}}

\newcommand{\sL}{\ensuremath{\mathsf L}}

\newcommand{\cD}{\ensuremath{\mathcal D}}

\newcommand{\cH}{\ensuremath{\mathcal H}}

\newcommand{\cL}{\ensuremath{\mathcal L}}

\newcommand{\cdist}[1]{\mathsf{c}_{#1}}

\newcommand{\com}{\ensuremath{\mathsf{K}}}

\newcommand{\dX}{\ensuremath{\mathsf{d_X}}}
\newcommand{\dY}{\ensuremath{\mathsf{d_Y}}}

\newcommand{\sd}{\ensuremath{\mathsf{d}}}

\newcommand{\laak}{\mathsf{L_k}}
\newcommand{\diak}{\mathsf{D_k}}

\newcommand{\diam}{\ensuremath{\mathrm{diam}}}

\newcommand{\met}[1]{\mathsf{#1}}

\newcommand{\gra}[1]{\mathsf{#1}}

\newcommand{\banX}{\ban X}
\newcommand{\banY}{\ban Y}
\newcommand{\metX}{\met X}
\newcommand{\metY}{\met Y}

\newcommand{\banYn}{(\ban Y,\norm{\cdot})}
\newcommand{\metXd}{(\met X,\dX)}
\newcommand{\metYd}{(\met Y,\dY)}

\newcommand{\eqd}{\stackrel{\mathrm{def}}{=}}

\newcommand{\ban}[1]{\mathfrak{#1}}

\newcommand{\lrpar}[1]{\left( #1\right)}

\newcommand{\norm}[1]{\| #1\|}
\newcommand{\bnorm}[1]{\Big\| #1\Big\|}

\begin{document}
	
	\title[No dimension reduction for doubling spaces revisited]{No dimension reduction for doubling subsets of $\ell_q$ when $q>2$ revisited}
	
	\author{F.~ Baudier}
	\address{Texas A\&M University, College Station, TX 77843, USA}
	\email{florent@math.tamu.edu}
	
	\author{K.~ Swie\c cicki}
	\address{Texas A\&M University, College Station, TX 77843, USA}
	\email{ksas@math.tamu.edu}
	
	\author{A.~ Swift}
	\address{University of Oklahoma, Norman, OK 73019, USA}
	\email{ats0@ou.edu}
	\date{}
	
	\thanks{F. Baudier was partially supported by the National Science
		Foundation under Grant Number DMS-1800322.}
	
	\keywords{dimension reduction, doubling spaces, bi-Lipschitz embedddings, distortion, Laakso and diamond graphs, non-positive curvature, asymptotic midpoint uniform convexity}
	\subjclass[2010]{46B85, 68R12, 46B20, 51F30, 05C63, 46B99}

	\begin{abstract}
		We revisit the main results from \cites{BGN_SoCG14,BGN_SIAM15} and \cite{LafforgueNaor14_GD} about the impossibility of dimension reduction for doubling subsets of $\ell_q$ for $q>2$. We provide an alternative elementary proof of this impossibility result that combines the simplicity of the construction in \cites{BGN_SoCG14,BGN_SIAM15} with the generality of the approach in \cite{LafforgueNaor14_GD} (except for $L_1$ targets). One advantage of this different approach is that it can be naturally generalized to obtain embeddability obstructions into non-positively curved spaces or asymptotically uniformly convex Banach spaces.
	\end{abstract}

	\maketitle
	
	\setcounter{tocdepth}{3}
	\section{Introduction}
	The celebrated Johnson-Lindenstrauss \cite{JL84} lemma asserts that any $n$-point subset of $\ell_2^n$ admits a bi-Lipschitz embedding with distortion at most $1+\vep$ into $\ell_2^k$ where $k=O(\frac{\log n}{\vep^2})$. This dimension reduction phenomenon is a fundamental paradigm as it can be used to improve numerous algorithms in theoretical computer science (cf. \cite{Naor_ICM18}) both in terms of running time and storage space. Johnson and Lindenstrauss observed that a simple volume argument gives that the dimension must be at least $\Omega(\log \log n)$. Later Alon \cite{Alon03} showed that the bound in the Johnson-Lindenstrauss lemma was tight up to a $\log(1/\vep)$ factor. Recently, Larsen and Nelson \cite{LarsenNelson17} were able to show the optimality of the dimension bound in the Johnson-Lindenstrauss lemma. A common feature of the subsets exhibiting lower bounds on the dimension is that they have high doubling constants. In \cite{LangPlaut01}, Lang and Plaut raised the following fundamental question. 
	
	\begin{prob}\label{prob:reductiondoubling}
		Does a doubling subset of $\ell_2$ admit a bi-Lipschitz embedding into a constant dimensional Euclidean space?
	\end{prob}
	
	Based on a linear programming argument, Brinkman and Charikar \cite{BrinkmanCharikar05} proved that there is no dimension reduction in $\ell_1$. An enlightening geometric proof was given by Lee and Naor in \cite{LN04}. The subset of $\ell_1$ that does not admit dimension reduction is the diamond graph $\diak$ and has a high doubling constant. However, there does exist a doubling subset\footnote{The results are asymptotic in nature and by doubling we mean that the doubling constant of $\laak$ is $O(1)$. The classical notation $D=O(f(n))$ (resp. $D=\Omega(f(n))$) means that $D\le \alpha f(n)$ (resp. $D\ge \alpha f(n)$) for some constant $\alpha$ and for $n$ large enough. And $D=\Theta(f(n))$ if and only if $[D=O(f(n))] \land [D=\Omega(f(n))]$.} of $\ell_1$, the Laakso graph $\laak$, for which existence of a bi-Lipschitz embedding with distortion $D$ into $\ell_1^d$ implies that $D=\Omega(\sqrt{\log(n)/\log(d)})$, or equivalently there is no bi-Lipschitz embedding of $\laak$ with distortion $D$ in $\ell_p^k$ if $k=O(n^{1/D^2})$. Therefore, Problem \ref{prob:reductiondoubling} has a negative solution for $\ell_1$-targets. That Problem \ref{prob:reductiondoubling} also has a negative solution for $\ell_q$-targets for $q>2$ was proved independently by Y. Bartal, L.-A. Gottlieb, and O. Neiman \cite{BGN_SoCG14,BGN_SIAM15}, and V. Lafforgue and A. Naor \cite{LafforgueNaor14_GD}.
	
	\begin{theo}\label{theo:dimreduction}
		For every $q\in(2,\infty)$, there exists a doubling subset of $\ell_q$ that does not admit any bi-Lipschitz embedding into $\bR^d$ for any $d\in\bN$.
	\end{theo}
	
	In Section \ref{sec:new-proof} we give a new proof of Theorem \ref{theo:dimreduction}.
	In order to put our contribution into perspective and to highlight the advantages and limits of our alternative proof,  we will discuss the two distinct approaches taken in \cite{BGN_SoCG14,BGN_SIAM15} and \cite{LafforgueNaor14_GD}, as well as their scopes of application.
	
	The approach undertaken by Lafforgue and Naor is based on classical, albeit subtle, geometric properties of Heisenberg groups. In \cite{LafforgueNaor14_GD}, Lafforgue and Naor construct for every $\vep\in(0,\frac12]$ and $q\in[2,\infty)$, an embedding $F_{\vep,q}\colon \bH_3(\bZ)\to L_q(\bR^s)$ such that $F_{\vep,q}(\bH_3(\bZ))$ is $2^{16}$-doubling and
	\begin{equation}\label{eq:LNeq}
		\forall x,y\in \bH_3(\bZ),\ \sd_{W}(x,y)^{1-\vep}\le \|F_{\vep,q}(x)-F_{\vep,q}(y)\|\lesssim \frac{\sd_W(x,y)^{1-\vep}}{\vep^{1/q}},
	\end{equation} 
	where $\sd_W$ is the canonical word metric on the discrete $3$-dimensional Heisenberg group $\bH_3(\bZ)$, and $\bR^s$ is some potentially high-dimensional Euclidean space equipped with the Lebesgue measure. The symbol $\lesssim$ will be conveniently use to hide a universal numerical mulitiplicative constant.
	
	The map $F_{\vep,q}$ is given by a rather elementary formula but showing that it is a bi-Lipschitz embedding of the $(1-\vep)$-snowflaking of $\bH_3(\bZ)$ as in \eqref{eq:LNeq}, and that the image is doubling requires some quite technical analytic computations \footnote{Lafforgue and Naor gave an alternate (and of similar difficulty) proof of \eqref{eq:LNeq} using the Schr\"odinger representation of Heisenberg groups that we do not discuss here.}. By taking $\vep=1/\log n$ in \eqref{eq:LNeq}, the map $F_{1/\log n,q}$ becomes a bi-Lipschitz embedding with distortion $O((\log n)^{1/q})$ of the ball of radius $\sqrt[4]{n}$ into $L_q$ (whose image inherits the doubling property of $F_{\vep,q}(\bH_3(\bZ))$). Since $\bH_3(\bZ)$ is a finitely generated group of quartic growth, for every $n\ge 1$ there exists a $n$-point subset $X_n\subset\bH_3(\bZ)$ lying in an annulus enclosed by two balls with radii proportional to $\sqrt[4]{n}$. The image of $X_n$ under $F_{1/\log n,q}$, which will be denoted $\cH_n(q)$, is $2^{16}$-doubling. A significant advantage of the Heisenberg-based approach of Lafforgue and Naor is that it provides non-embeddability results for the doubling subset $\cH_n(q)$ of $\ell_q$ for a wide class of Banach space targets. It is indeed possible to leverage some deep non-embeddability results available for the subset $X_n$ of $\bH_3(\bZ)$, to derive lower bounds on the distortion of $\cH_n(q)$ when embedding $\cH_n(q)$ into any $p$-uniformly convex Banach space for $2\le p<q$ and even into $L_1$. 
	
	Let $\cdist{\metY}(\metX)$ denote the $\metY$-distortion of $\metX$ for two metric spaces $\metXd$ and $\metYd$. The following theorem is a quantitatively explicit version (and updated according to the most recent available bounds) of Theorem~1.2 in \cite{LafforgueNaor14_GD}.
	
	\begin{theo}\label{theo:LN}
		For every $q\in(2,\infty)$ and every $n\in\bN$, there exists a $2^{16}$-doubling $n$-point subset $\cH_n(q)$ of $\ell_q$ such that 
		\begin{enumerate}
			\item\label{stat:LN1} $\cdist{\banY}(\cH_n(q))=\Omega((\log n)^{\frac1p-\frac1q})$ if $\banY$ is a $p$-uniformly convex Banach space for $2\le p<q$
			\item\label{stat:LN2} $\cdist{L_1}(\cH_n(q))=\Omega((\log n)^{\frac14})$.
		\end{enumerate}
		Moreover, for every $q\in(2,\infty)$, there exists a doubling subset $\cH(q)$ of $\ell_q$ that does not admit a bi-Lipschitz embedding into $L_1$ or into a $p$-uniformly convex Banach space for any $p\in(1,q)$.
	\end{theo}
	
	Assertions $\eqref{stat:LN1}$ and $\eqref{stat:LN2}$ in Theorem \ref{theo:LN} follow from the above discussion of the Lafforgue-Naor approach and sharp non-embeddability of Heisenberg balls into $p$-uniformly convex spaces (\cite{LafforgueNaor14_IJM}, which refines earlier results from \cite{ANT13}), and into $L_1$ (more specifically \cite{NaorYoung18}, which improves the lower bound in \cite{CKN11}). The moreover part of Theorem \ref{theo:LN} follows from a standard argument where $\cH(q)$ is a certain disjoint union of the sequence $\{\cH_n(q)\}_{n\in\bN}$ and which contains an isometric copy of a rescaling of $\cH_n(q)$ for every $n\in \bN$. 
	
	The derivation of Theorem \ref{theo:dimreduction} from Theorem \ref{theo:LN}, which follows from the fact that we can assume without loss of generality that the constant finite-dimensional space is $2$-uniformly convex, is standard. Another consequence of assertion $\eqref{stat:LN1}$ in Theorem \ref{theo:LN} and classical estimates on the Banach-Mazur distance between finite-dimensional $\ell_r$-spaces is the following corollary.
	
	\begin{coro}\label{coro:LN}
		For every $q\in(2,\infty)$ and every $n\in\bN$, there exists a $2^{16}$-doubling $n$-point subset $\cH_n(q)$ of $\ell_q$ such that 
		\begin{enumerate}
			\item $\cdist{\ell_q^d}(\cH_n(q))=\Omega\Big(\Big(\frac{\log n}{d}\Big)^{\frac12-\frac1q}\Big)$.
			\item $\cdist{\ell_p^d}(\cH_n(q))=\Omega\Big(\Big(\frac{\log n}{d}\Big)^{\min\big\{\frac12,\frac1p\big\}-\frac1q}\Big)$ if $1<p<q$.
		\end{enumerate}
	\end{coro}
	
	It is worth pointing out that the case $q=2$ in assertion $\eqref{stat:LN1}$ of Theorem \ref{theo:LN} also follows from an important Poincar\'e-type inequality for the Heisenberg group \cite[Theorem 1.4 and Corollary 1.6]{ANT13} which is a precursor of a groundbreaking line of research pertaining to Poincar\'e-type inequalities in terms of horizontal versus vertical perimeter in Heisenberg groups. 
	
	We now turn to the approach of Bartal, Gottlieb, and Neiman.
	
	\begin{theo}\label{theo:BGN}\cite{BGN_SIAM15}
		Let $q\in(2,\infty)$, $D\ge 1$, and $d\in\bN$. For every $n\in\bN$ there exists a $n$-point subset $\cL_n(p,q,D,d)$ of $\ell_q$ that is $2^{32}$-doubling and such that any bi-Lipschitz embedding of $\cL_n(p,q,D,d)$ with distortion $D$ into $\ell_p^d$ must satisfy 
		\begin{enumerate}
			\item\label{stat:BGN1} $D=\Omega\Big(\Big(\frac{\log n}{d}\Big)^{\frac12-\frac1q}\Big)$ if  $p=q$
			\item[] and
			\item\label{stat:BGN2} $D=\Omega\left(\frac{(\log n)^{\frac12-\frac1q}}{d^{\frac{\max\{p-2,2-p\}}{2p}}}\right)$ if $1\le p<q$.\end{enumerate}
	\end{theo}
	
	A conceptual difference between Theorem \ref{theo:BGN} and Corollary \ref{coro:LN} is that in Theorem \ref{theo:BGN} the finite doubling subsets depend on the distortion, the dimension, and also the host space. Consequently, the sequence $\{\cL_n(p,q,D,d)\}_{n\ge 1}$ only rules out bi-Lipschitz embeddings for fixed distortion and dimension. Nevertheless, one can still derive Theorem \ref{theo:dimreduction} from Theorem \ref{theo:BGN}. This derivation, which was omitted in \cite{BGN_SoCG14,BGN_SIAM15}, will be recalled at the end of Section \ref{sec:new-proof}. The doubling subset $\cL_n(p,q,D,d)$ of $\ell_q$ is based on an elementary construction of a $\Theta(6^k)$-point Laakso-like structure in $\ell_q^k$ that we will recall in Section \ref{sec:new-proof} since our new proof of Theorem \ref{theo:dimreduction} uses the same construction. The combinatorial proof of Theorem \ref{theo:BGN} in \cite{BGN_SIAM15} utilizes a newly introduced method based on potential functions, i.e. functions of the form $\Phi_{p,q}(u,v)=\frac{\|f(u)-f(v)\|_p^p}{\|u-v\|_q^p}$ for some $p,q$, where $\{u,v\}$ is an ``edge" of $\cL_n(p,q,D,d)$. The method of potential functions relies heavily on the fact that every map taking values into $\ell_p^d$ can be decomposed as a sum of $d$ real-valued (coordinate) maps, and this method does not seem to be easily extendable to more general Banach space targets. 
	
	In Section \ref{sec:new-proof} we present a new proof of Theorem \ref{theo:dimreduction}. The doubling subsets are identical to the ones of Bartal-Gottlieb-Neiman and they are described in Section \ref{sec:thin}. The proof uses a self-improvement argument, which was first employed for metric embedding purposes by Johnson and Schechtman in \cite{JS09}, and subsequently in \cite{Kloeckner14}, \cite{BaudierZhang16}, \cite{Baudier-et-all-JFA}, \cite{Swift18}, and \cite{Zhang}; and is carried over in Section \ref{sec:self-improv}. Our proof has several advantages. We prove an analog of Theorem \ref{theo:BGN} where the $n$-point doubling subset can be chosen independently of the dimension and improve the estimates in assertion \ref{stat:BGN2}. Moreover, the self-improvement approach is rather elementary and yet covers the case of uniformly convex target spaces as in the work of Lafforgue-Naor. However, it does not allow the recovery of the case of an $L_1$ target as in assertion $\eqref{stat:LN2}$ of Theorem \ref{theo:LN}.  The fact that we will be dealing with abstract metric structures that are not graph metrics requires a significantly more delicate implementation of the self-improvement argument. In Section \ref{sec:lp-uc} we explain how the new proof allows us to derive known tight lower bounds for the distortion of $\ell_p^n$ into uniformly convex spaces. It is worth mentioning that the lower bounds that can be derived from the Bartal-Gottlieb-Neiman approach and the Lafforgue-Naor approach seem to be often suboptimal. In Section \ref{sec:npc} we extend the technique to cover purely metric targets of non-positive curvature and more generally rounded ball metric spaces. Finally, in Section \ref{sec:auc} we extend our approach to the asymptotic Banach space setting. For this purpose, we construct countably branching analogs of the structures introduced by Bartal, Gottlieb, and Neiman that provide quantitative obstructions to embeddability into asymptotically midpoint uniformly convex spaces.

	\section{Impossibility of dimension reduction in $\ell_q$, $q>2$}
	\label{sec:new-proof}
	
	\subsection{Thin Laakso substructures}
	\label{sec:thin}
	Let us recall first a procedure to construct recursively certain sequences of graphs such as the classical diamond graphs $\{\diak\}_{k\in\bN}$ and Laakso graphs $\{\laak\}_{k\in\bN}$, and their countably branching analogues $\{\sD_{\text{k}}^{\omega}\}_{k\in\bN}$ and $\{\sL_{\text{k}}^{\omega}\}_{k\in\bN}$. 
	
	A directed $s$-$t$ graph $\gra G=(V,E)$ is a directed graph which has two distinguished vertices $s,t\in V$. To avoid confusion, we will also write sometimes $s(\gra G)$ and $t(\gra G)$. There is a natural way to ``compose'' directed $s$-$t$ graphs using the $\oslash$-product defined in \cite{LeeRag10}. Informally, the $\oslash$ operation replaces all the edges of an $s$-$t$ graph by identical copies of a given $s$-$t$-graph. Given two directed $s$-$t$ graphs $\gra H$ and $\gra G$, define a new graph $\gra H\oslash \gra G$ as follows:
	\begin{enumerate}[i)]
		\item $V(\gra H\oslash \gra G) \eqd V(\gra H)\cup(E(\gra H)\times(V(\gra G)\backslash\{s(\gra G),t(\gra G)\}))$
		\item For every oriented edge $e=(u,v)\in E(\gra H)$, there are $|E(\gra G)|$ oriented edges,
		\begin{align*}
			&\big\{(\{e,v_1\},\{e,v_2\})\mid (v_1,v_2)\in E(\gra G) \text{ and }v_1,v_2\notin \{s(\gra G),t(\gra G)\}\big\}\\
			\cup &\big\{(u,\{e,w\}) \mid (s(\gra G),w)\in E(\gra G)\big\}\cup \big\{(\{e,w\},u) \mid (w, s(\gra G))\in E(\gra G)\big\}\\
			\cup& \big\{(\{e,w\},v) \mid (w,t(\gra G))\in E(\gra G)\big\}\cup \big\{(v,\{e,w\}) \mid (t(\gra G),w)\in E(\gra G)\big\}
		\end{align*}
		\item $s(\gra H\oslash \gra G) \eqd s(\gra H)$ and $t(\gra H\oslash \gra G) \eqd t(\gra H)$.
	\end{enumerate}

	It is clear that the $\oslash$-product is associative (in the sense of graph-isomorphism or metric space isometry), and for a directed graph $\gra G$ one can recursively define $\gra G^{\oslash^{k}}$ for all $k\in\bN$ as follows:
	\begin{itemize}
		\item $\gra G^{\oslash^{1}} \eqd \gra G$.
		\item $\gra G^{\oslash^{k+1}} \eqd \gra G^{\oslash^{k}}\oslash \gra G$, for $k\ge 1$.  
	\end{itemize}
	
	Note that it is sometimes convenient, for some induction purposes, to define $\gra G^{\oslash^{0}}$ to be the two-vertex graph with an edge connecting them. Note also that if the base graph $\gra G$ is symmetric the graph $\gra G^{\oslash^{k}}$ does not depend on the orientation of the edges.
	
	\medskip
	If one starts with the $4$-cycle $\gra C_4$, the graph $\diak \eqd \gra{C}_4^{\oslash^{k}}$ is the diamond graph of depth $k$. The countably branching diamond graph of depth $k$ is defined as $\sD_k^\omega \eqd \com_{2,\omega}^{\oslash^{k}}$, where $\com_{2,\omega}$ is the complete bipartite infinite graph with two vertices on one side, (such that one is $s(\com_{2,\omega})$ and the other $t(\com_{2,\omega})$), and countably many vertices on the other side. The Laakso graph $\sL_k\eqd \sL_1^{\oslash^{k}}$ where the base graph $\sL_1$ is the graph depicted below. 
	
	The Laakso graphs do not admit bi-Lipschitz embeddings into any uniformly convex Banach space, in particular into $\ell_p$ when $p\in(1,\infty)$, and this is due to the fact that there are, at all scales, midpoints that are far apart. The idea of Bartal, Gottlieb, and Neiman was to slightly tweak the Laakso construction by reducing the distance between the midpoints so that these modified metric structures could fit into $\ell_p^k$ for arbitrarily large dimension $k$ but not into $\ell_p^d$ for fixed $d$ without incurring a large distortion. It will be convenient to abstract the construction of Bartal, Gottlieb, and Neiman and to that end, we introduce the following definition.
	
	\begin{defi}[Thin Laakso substructure]
		Let $q\in[1,\infty]$ and $\vep>0$. For $k\in \bN$, we say that a metric space $\metXd$ admits a \emph{$(\vep,q)$-thin $k$-Laakso substructure}, if there exists a collection of points $\cL_k(\vep,q)\subset \metX$ indexed by $\sL_k$ (and we will   identify the points in $\cL_k(\vep,q)$ with the corresponding points in $\cL_k$) such that for every $1\le j\le k$ and for all $\{s, a, m_1, m_2, b, t\}\subset \cL_k(\vep,q)$ indexed by any copy of the Laakso graph $\sL_1$ created at level $j$, the following interpoint distance equalities hold:
		\begin{enumerate}
			\item[($c_1$)] $\dX(s,a)=\dX(b,t)=\frac12 \dX(a,b)=\frac14\dX(s,t)>0$
			\item[($c_2$)] $\dX(s,b)=\dX(a,t)=\frac34\dX(t,s)$
			\item[($c_3$)]$\dX(m_1,a)=\dX(m_1,b)=\dX(m_2,a)=\dX(m_2,b)=\frac14(1+(2\vep)^q)^{1/q}\dX(s,t)$
			\item[($c_4$)]\label{eq:c4}$\dX(s,m_1)=\dX(m_2,s)=\dX(m_1,t)=\dX(m_2,t)=\frac12(1+\vep^q)^{1/q}\dX(s,t)$
			\item[($c_5$)]$\dX(m_1,m_2)=\vep \cdot \dX(s,t)$ (midpoint separation).
		\end{enumerate}
	\end{defi}
	
	The distances in the combinatorial Laakso graph, which is the template for the construction, satisfy $(c_1)-(c_4)$ with $\vep=0$, and $(c_5)$ with $\vep=\frac12$, and the distances for a path graph with 4 points would satisfy $(c_1)-(c_5)$ with $\vep=0$. The following diagram can help visualize the differences between the distances in the Laakso graph $\sL_1$ and the $(\vep,q)$-thin $1$-Laakso substructure $\cL_1(\vep,q)$ construction.
	
	\begin{figure}[H]
		\label{fig:Laakso}
		\caption{Distances in Laakso graph $\sL_1$ and distances in $\vep$-thin Laakso structure $\cL_1(\vep,q)$ in $\ell_q^2$}
		\vskip 0.2cm
		\hskip 2cm\includegraphics[scale=.7]{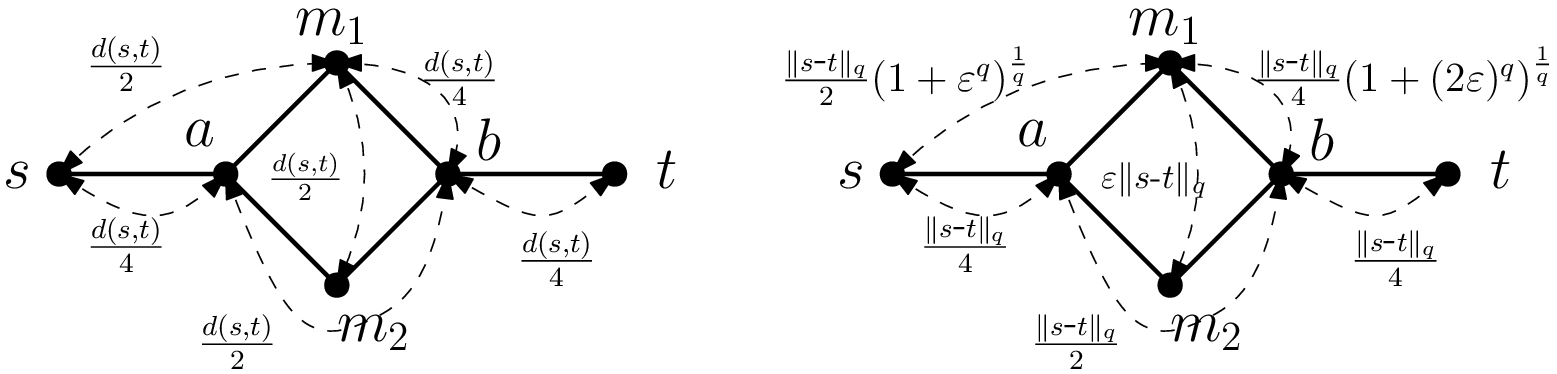}
		\end{figure}
	
	The existence of $(\vep,q)$-thin $k$-Laakso substructures in $\ell_q$ was proven in \cite{BGN_SIAM15}. Since we use different notation and a slightly different thinness parameter we will reproduce the proof for the convenience of the reader.
	
	\begin{lemm}\label{lemm:thinLaakso}
		Let $q\in[1,\infty]$. For all $k\in\bN$ and $\vep>0$, $\ell_q^{k+1}$ admits a $(\vep,q)$-thin $k$-Laakso substructure.
	\end{lemm}
	
	\begin{proof}
		Let $\{e_i\}_{i=1}^{k+1}$ be the canonical basis of $\ell_q^{k+1}$. The proof is by induction on $k$. If $k=1$ then $\sL_1=\{s,a,m_1,m_2,b,t\}$ and identifying points in $\cL_1(\vep,q)$ with the corresponding points in $\cL_1$ we define
		\begin{align*}
			s=-e_1 \quad and \quad  t=e_1,\\
			a=-\frac12 e_1 \quad  and \quad  b=\frac12 e_1,\\
			m_1= \vep e_2 \quad and \quad m_2=-\vep e_2.
		\end{align*}
		Observe that the vectors are in $\ell_p^2$ and a straightforward verification shows that conditions $(c_1)-(c_5)$ are verified. Assume now that $\cL_{k}(\vep,q)$ has been constructed in $\ell_q^{k+1}$. Recall that $\sL_{k+1}$ is contructed by replacing every edge in $\sL_k$ with a copy of $\sL_1$. For every edge $\{s,t\}$ in $\sL_k$ we introduce $4$ new points as follows:
		\begin{align*}
			a=\frac34 s+\frac14 t  \quad  and \quad  b=\frac14 s+\frac34 t,\\
			m_1= \frac{s+t}{2}+\frac{\vep}{2}\norm{s-t}_q e_{k+2} \quad and \quad m_2=\frac{s+t}{2}-\frac{\vep}{2}\norm{s-t}_q e_{k+2}.
		\end{align*}
		Then 
		\begin{align*}
			\norm{b-m_2}_q & = & \bnorm{\frac14 s+\frac34 t-\frac{s+t}{2}+\frac{\vep}{2}\norm{s-t}_q e_{k+2}}_q =  \bnorm{\frac{t-s}{4}+\frac{\vep}{2}\norm{s-t}_q e_{k+2}}_q\\
			& = &\Big(\frac{1}{4^q}\norm{s-t}_q^q+\frac{\vep^q}{2^q}\norm{s-t}_q^q\Big)^{1/q} = \frac{\norm{s-t}_q}{4}\Big(1+(2\vep)^q\Big)^{1/q},
		\end{align*}
		where in the penultimate equality we used the fact that $\frac14(s-t)\in \ell_q^{k+1}$. The other equalities can be checked similarly.
	\end{proof}
	
	\begin{rema}
		It was proved in \cite{BGN_SIAM15} that an $(\vep,p)$-thin $k$-Laakso substructure is $2^{32}$-doubling whenever $\vep<\frac{2}{17}$.
	\end{rema}
	
	\subsection{A proof via a self-improvement argument}
	\label{sec:self-improv}
	
	In this section we prove Theorem \ref{theo:dimreduction} using a self-improvement argument.  Recall that a Banach space $\banX$ is \emph{uniformly convex} if for all $t>0$ there exists $\delta(t)>0$ such that for all $x,y\in S_X$, if $\norm{x-y}_\banX\ge t$ then $\norm{\frac{x+y}{2}}_\banX\le 1-\delta(t)$. The modulus of uniform convexity of $\banX$, denoted $\delta_{\banX}$, is defined by 
	\begin{equation}
		\delta_{\banX}(t)=\inf\left\{1-\left\|\frac{x+y}{2}\right\|_\banX\colon \|x-y\|_\banX\ge t\right\}.
	\end{equation}
	Clearly, $\banX$ is uniformly convex if and only if $\delta_{\banX}(t)>0$ for all $t>0$, and we say that $\banX$ is $q$-uniformly convex (or is uniformly convex of power type $q$) if $\delta_\banX(t)\ge ct^{q}$ for some universal constant $c>0$. A classical result of Pisier \cite{Pisier75} states that a uniformly convex Banach space admits a renorming that is $q$-uniformly convex for some $q\ge 2$. The following key lemma is similar to a contraction result for Laakso graphs from \cite{JS09}.
	
	\begin{lemm}\label{lemm:contraction_uc}
		Let $p\in(1,\infty)$. Assume that $\cL_k (\vep,p)$ is a $(\vep,p)$-thin  $k$-Laakso substructure in $\metXd$ and that $f\colon \metX \to \banYn$ is a bi-Lipschitz embedding with distortion $D$. Then for every $1\le \ell \le k$, if $\{s,a,m_1,m_2,b, t\}\subset \sL_k(\vep,p)$ is indexed by a copy of one of the Laakso graphs $\sL_1$ created at step $\ell$, we have:
		\begin{equation}
			\|f(s)-f(t)\|\le D\dX(s,t)(1+\vep^p)^{1/p}\left(1-\delta_\banY\left(\frac{2\vep}{D(1+\vep^p)^{1/p}}\right)\right).
		\end{equation} 
	\end{lemm}
	
	\begin{proof}
		Assume without loss of generality that for all $x,y\in \metX$ 
		\begin{equation}\label{eq:Lip}
			\dX(x,y)\le \norm{f(x)-f(y)}\le D\dX(x,y).
		\end{equation} 
		Let $\alpha \eqd \frac{\dX(s,t)}{2}(1+\vep^p)^{1/p}$, and  $$x_1 \eqd f(m_1)-f(s), \quad x_2 \eqd f(m_2)-f(s), \quad y_1 \eqd f(t)-f(m_1), \quad y_2 \eqd f(t)-f(m_2).$$ For all $i\in\{1,2\}$, it follows from the upper bound in \eqref{eq:Lip} and $(c_4)$ that $\frac{\norm{x_i}}{D\alpha}\le 1$ and $\frac{\norm{y_i}}{D\alpha} \le 1$. On the other hand, it follows from the lower bound in \eqref{eq:Lip} and $(c_5)$ that 
		$$\frac{\norm{x_1-x_2}}{D\alpha}\ge \frac{2\vep}{D(1+\vep^p)^{1/p}} \quad \textrm{ and } \quad  \frac{\norm{y_1-y_2}}{D\alpha}\ge \frac{2\vep}{D(1+\vep^p)^{1/p}}.$$
		
		Therefore  
		
		$$\bnorm{\frac{x_1+x_2}{2D\alpha}}\le 1-\delta_\banY\Big(\frac{2\vep}{D(1+\vep^p)^{1/p}}\Big)\quad \textrm{ and } \quad \bnorm{\frac{y_1+y_2}{2D\alpha}}\le1-\delta_\banY\Big(\frac{2\vep}{D(1+\vep^p)^{1/p}}\Big).$$ 
		
		Since 
		
		\begin{align*}
			f(t)-f(s) & = (f(t)-f(m_1)+f(m_1)-f(s)+f(t)-f(m_2)+f(m_2)-f(s))/2\\
			& = (y_1+y_2+x_1+x_2)/2,
		\end{align*}
		
		it follows from the triangle inequality that 
		$\|\frac{f(t)-f(s)}{D\alpha}\|\le2\Big(1-\delta_\banY\Big(\frac{2\vep}{D(1+\vep^p)^{1/p}}\Big)\Big)$ and the conclusion follows.
	\end{proof}
	
	By using the tension between the thinness parameter of a thin Laakso substructure and the power type of the modulus of uniform convexity of the host space we can prove the following distortion lower bound.
	
	\begin{theo}\label{theo:main1}
		Let $2\le p<q$ and assume that $\metXd$ admits a bi-Lipschitz embedding with distortion $D$ into a $p$-uniformly convex Banach space $\banY$. There exists $\vep:=\vep(p,q,D,\banY)>0$ such that if $\metXd$ admits a $(\vep,q)$-thin $k$-Laakso substructure then $D=\Omega(k^{1/p-1/q})$
	\end{theo}

	\begin{proof}
		Assume that for all $x,y\in \metX$ 
		\begin{equation}\label{eq:}
			\dX(x,y)\le \norm{f(x)-f(y)}\le D\dX(x,y),
		\end{equation} 
		and let $\cL_k(\vep,q)$ be a $(\vep,q)$-thin $k$-Laakso substructure with $\vep>0$  small enough such that $(1+\vep^q)^{1/q}\le 2$. The self-improvement argument uses the self-similar structure of the Laakso graphs. For $1\le j\le k$ consider the decomposition $\sL_{k-j}\oslash \sL_j$ of $\sL_k$, i.e. $\sL_k$ is formed by replacing each of the $6^{k-j}$ edges of $\sL_{k-j}$ by a copy of $\sL_j$. We define $D_j$ to be the smallest constant such that 
		\begin{equation}
			\norm{f(x)-f(y)}\le D_j \dX(x,y),
		\end{equation} 
		for all $4\times 6^{k-j}$ pairs of points $\{x,y\}$ in $\cL_k(\vep,p)$ that are indexed by vertices of a copy of $\sL_j$ in $\sL_k$ of the form $\{s,m_i\}$ or $\{m_i,t\}$ for some $i\in\{1,2\}$, where $s$ and $t$ are the farther apart vertices in $\sL_j$ whose two distinct midpoints are $m_1$ and $m_2$. 
		
		It is clear that for all $j\in\{1,\dots,k\}$, the inequalities $1\le D_j\le D$ hold. Assume that $\delta_\banY(t)\ge ct^{p}$ for some constant $c>0$ (that depends on $\banY$ only). Fix $\sL^0_j$ as one of the $6^{k-j}$ copies of $\sL_j$ in the decomposition $\sL_{k-j}\oslash\sL_j$ of $\sL_k$. Observe that $\sL^0_j=\sL_1\oslash \sL_{j-1}$ and let $\{s,a,m_1,m_2,b,t\}$ denote the vertices of $\sL_1$ in this decomposition of $\sL_j^0$. Consider the pair $\{s, m_1\}$ as defined above (the 3 other pairs can be treated similarly) and the two copies of $\sL_{j-1}$ which contain either $s$ or $m_1$ and have the vertex $a$ in common. In the proof of Lemma \ref{lemm:contraction_uc} we only used the upper bound in $\eqref{eq:Lip}$ for pairs of points of the form described in the definition of $D_{j-1}$, and because we assumed that $\delta_\banY(t)\ge ct^{p}$ and $(1+\vep^q)^{1/q}\le 2$ we have 
		
		\begin{equation}
			\|f(s)-f(a)\|\le D_{j-1}\dX(s,a)(1+\vep^q)^{1/q}\Big(1-\frac{c\vep^p}{D_{j-1}^p}\Big),
		\end{equation} 
		and
		\begin{equation}
			\|f(a)-f(m_1)\|\le D_{j-1}\dX(a,m_1)(1+\vep^q)^{1/q}\Big(1-\frac{c\vep^p}{D_{j-1}^p}\Big).
		\end{equation}
		
		Then, it follows from the triangle inequality that
		
		\begin{equation}\label{eq2:}
			\|f(s)-f(m_1)\|\le D_{j-1}(\dX(s,a)+\dX(a,m_1))(1+\vep^q)^{1/q}\Big(1-\frac{c\vep^p}{D_{j-1}^p}\Big).
		\end{equation}
		
		By $(c_1)$ and $(c_3)$ in the construction of the thin Laakso substructures we have 
		$$\dX(s,a)+\dX(a,m_1)=\frac14 \dX(s,t)+\frac14(1+(2\vep)^q)^{1/q}\dX(s,t).$$
		
		Since $q\ge 1$ we have $(1+(2\vep)^q)^{1/q}\le 1+(2\vep)^q$, and thus
		
		$$\dX(s,a)+\dX(a,m_1)\le \frac14 \dX(s,t)(2+(2\vep)^q)=\frac12 \dX(s,t)(1+2^{q-1}\vep^q)\stackrel{(c_4)}{=}\frac{1+2^{q-1}\vep^q}{(1+\vep^q)^{1/q}}\dX(s,m_1).$$
		
		Substituting this last inequality in \eqref{eq2:} we obtain
		
		\begin{equation*}
			\|f(s)-f(m_1)\|  \le  D_{j-1}\dX(s,m_1)(1+2^{q-1}\vep^q)\Big(1-\frac{c\vep^p}{D_{j-1}^p}\Big).
		\end{equation*}

		By symmetry of the $(\vep,q)$-thin Laakso substructures, the other pairs of points in the definition of $D_j$ can be treated similarly and hence we have proved that 
		
		\begin{equation*}\label{eq:induc}
			D_{j}\le D_{j-1}(1+2^{q-1}\vep^q)\Big(1-\frac{c\vep^p}{D_{j-1}^p}\Big).
		\end{equation*}
		
		Then,
		
		\begin{align}
			\nonumber D_{j} & \le  D_{j-1}(1+2^{q-1}\vep^q)-\frac{(1+2^{q-1}\vep^q)c\vep^{p}}{D_{j-1}^{p-1}}\\
			\label{eq:10}	   & \le  D_{j-1} +D(2\vep)^q-\frac{c\vep^{p}}{D^{p-1}},
		\end{align}
		where in \eqref{eq:10} we used the fact that $D_{j-1}\le D$ and $1+2^{q-1}\vep^q\ge 1$.
		Rearranging we have
		\begin{equation}
			D_{j-1}-D_{j} \ge D\left(\frac{c\vep^p}{D^p}-(2\vep)^q\right),
		\end{equation}
		If we let $\vep=\gamma D^{-\frac{p}{q-p}}$ for some small enough $\gamma$ to be chosen later (and that depends only on $p,q$, and $c$), then 
		\begin{align}
			D_{j-1}-D_{j} & \ge  D\left(c\gamma^p\frac{D^{-\frac{p^2}{q-p}}}{D^p}-(2\gamma)^q D^{-\frac{pq}{q-p}}\right)\\
			& \ge  D\cdot D^{-\frac{pq}{q-p}}(c\gamma^p-2^q\gamma^q).
		\end{align}
		If we choose $\gamma\in\Big(0,\Big(\frac{c}{2^{q+1}}\Big)^{1/(q-p)}\Big)$, and since $p<q$, we have $c\gamma^p-2^q\gamma^q\ge \frac c2\gamma^p>0$. Hence $D_{j-1}-D_{j} \ge \frac{c\gamma^p}{2}D^{1-\frac{pq}{q-p}}$ and summing over $j=2,\dots, k$ we get 
		\begin{equation}
			D\ge D_1-D_k\ge \sum_{j=2}^{k}\frac{c\gamma^p}{2}D^{1-\frac{pq}{q-p}}\ge \frac{c\gamma^p}{2}(k-1)D^{1-\frac{pq}{q-p}},
		\end{equation}
		and hence
		$D\gtrsim k^{1/p-1/q}$.
	\end{proof}
	
	Corollary \ref{cor:BSS} below improves Theorem \ref{theo:BGN} in several ways. The dependence in the dimension for the thinness parameter is removed. Assertion \ref{stat:BSS1} extends to all $p$-uniformly convex Banach spaces the bound in assertion \ref{stat:BGN2} of Theorem \ref{theo:BGN} while improving the bound. Indeed if $\banY=\ell_p^d$ then $D=\Omega\Big(\big(\frac{\log n}{d}\big)^{\min\big\{\frac12,\frac1p\big\}-\frac1q}\Big)$.
	
	\begin{coro}\label{cor:BSS}
		Let $q\in(2,\infty)$, $\banY$ be a Banach space, and fix $D\ge 1$. For every $n\in\bN$ there exists an $n$-point subset $\cL_n(q,D,\banY)$ of $\ell_q$ that is $2^{32}$-doubling and such that any bi-Lipschitz embedding with distortion $D$ into $\banY$ must incur
		\begin{enumerate}
			\item\label{stat:BSS1} $D=\Omega \Big((\log n)^{\frac1p-\frac1q}\Big)$ if $p\in[2,q)$  and $\banY$ is a $p$-uniformly convex Banach space
			\item[] and
			\item\label{stat:BSS2} $D=\Omega \Big(\Big(\frac{\log n}{d}\Big)^{\frac12-\frac1q}\Big)$ if $\banY=\ell_q^d$
		\end{enumerate}
	\end{coro}
	
	\begin{proof}
		Assertion $(1)$ follows immediately from Theorem \ref{theo:main1} and Lemma \ref{lemm:thinLaakso}. The second assertion follows from the fact that $n=\Theta(6^k)$ and that the Banach-Mazur distance between the $2$-uniformly convex spaces $\ell_2^d$ and $\ell_q^d$ is at most $d^{1/2-1/q}$.
	\end{proof}

	\begin{rema} Very recently, Naor and Young \cite{NaorYoung20} gave the first partial counter-example to the metric Kadec-Pe\l czy\'{n}ski problem, which asks whether for $1\le p<r<q<\infty$, a metric space that admits a bi-Lipschitz embedding into $L_p$ and into $L_q$ necessarily admits a bi-Lipschitz embedding into $L_r$. Naor and Young produced a Heisenberg-type space that does embed into $\ell_1$ and into $\ell_q$ but does not embed into $\ell_ r$ for any $1<r<4\le q$. The fact that what happens for $\ell_r$ in the range $4\le r<q$ is not understood seems inherent of the Heisenberg approach. If we could show that the thin Laakso substructures do embed into $\ell_1$ then we would have a second counter-example to the metric Kadec-Pe\l czy\'{n}ski problem which resolves this issue.
	\end{rema}
	
	It remains to show how Theorem \ref{theo:dimreduction} can be derived from Corollary \ref{cor:BSS}. First observe that for all $q>2$, $D\ge 1$, and every $n\in\bN$ the $n$-point doubling subsets $\cL_n(q,D,\ell_2)$ of $\ell_q$ belong to the unit ball of $\ell_q$. Now consider the subset $Z_q\eqd\bigcup_{(k,n)\in\bN^2} \cL_n(q,k,\ell_2)\times \{(4^{k},4^{n})\}\subset \ell_q\oplus_q \bR^2\equiv \ell_q$. Clearly, $Z_q$ contains an isometric copy of $\cL_n(q,k,\ell_2)$ and it can be verified that $Z_q$ is doubling. If $Z_q\subset \ell_q$ admits a bi-Lipschitz embedding with distortion $D$ into $\ell_q^d$ for some $d\in\bN$, then  
	the proof of assertion \eqref{stat:BSS2} in Corollary \ref{cor:BSS} shows that $D=\Omega \Big(\Big(\frac{\log n}{d}\Big)^{\frac12-\frac1q}\Big)$
	since $Z_q$ contains an isometric copy of $\cL_n(q,k,\ell_2)$ for all $n\in \bN$, where $k\in\bN$ is such that $k\le D< k+1$, and hence $D$ cannot be finite.
	
	\subsection{Quantitative embeddability of $\ell_q^k$ into uniformly convex Banach spaces}
	\label{sec:lp-uc}
	
	The alternative proof of Theorem \ref{theo:dimreduction} that we proposed has several noteworthy applications. One application concerns the non-embeddability of $\ell_q$ into $L_p$ and more generally lower bounding the quantitative parameter $\cdist{\banY}(\ell_q^k)$ whenever $\banY$ is a $p$-uniformly convex Banach space. It is well known that when $2\le p<q$, if $\banY$ has cotype $p$, and in particular if $\banY$ is $p$-uniformly convex, then $\sup_{k\ge 1}\cdist{\banY}(\ell_q^k)=\infty$. Quantitatively,
	\begin{equation}\label{eq:lq-lp}
		\cdist{\banY}(\ell_q^k)=\Omega\Big(k^{\frac1p-\frac1q}\Big) \text{ for all }2\le p<q.
	\end{equation}
	and 
	\begin{equation}\label{eq:lq-lp<2}
		\cdist{\banY}(\ell_q^k)=\Omega\Big(k^{\frac12-\frac1q}\Big) \text{ for all }1\le p\le 2<q.
	\end{equation}
	The fact that these lower bounds are tight follows from simple estimates of the norm of the formal identity (and its inverse) (e.g. $\|I_{\ell_q^k\to\ell_p^k}\|\cdot \|I^{-1}_{\ell_q^k\to\ell_p^k}\|\le k^{1/p-1/q}$ is a consequence of H\"older's inequality and the monotonicity of the $\ell_r$-norms). Since the thin $k$-Laakso substructure lives in $\ell_q^{k+1}$ and has  $\Theta(6^k)$ points these lower bounds follow directly from the first assertion of Corollary \ref{cor:BSS}. As we point out below the approaches in \cite{LafforgueNaor14_GD} and \cite{BGN_SIAM15} seem to only give suboptimal results. 
	
	The $n$-point doubling subset $\cH_n(q)$ of Lafforgue and Naor lies in some $L_q(\bR^k)$-space and hence in $\ell_q^{n(n-1)/2}$ by a result of \cite{Ball90}. Therefore, if one uses the Heisenberg-type $\sqrt{n}$-point doubling subset of $\ell_q^n$ one can derive that, for example, $\cdist{\ell_2}(\ell_q^n)=\Omega\Big(\log (n)^{\frac12-\frac1q}\Big)$ which is suboptimal. To get the optimal lower bound one would need to be able to show that the doubling subset can actually be embedded into $\ell_q^{\Theta(\log n)}$, which is the best we can hope for due to assertion \ref{stat:LN1} in Theorem \ref{theo:LN}. This does not seem to be known and we do not know if this is true.
	
	Following Bartal-Gottlieb-Neiman's approach, one could obtain partial and suboptimal results as follows. Assume that $\ell_q^k$ admits a bi-Lipschitz embedding into $\ell_p^k$. Then one can construct a subset in $\ell_q^{k}$, namely $\cL_n(q,D,k-1)$, having $n=\Theta(6^k)$ points and witnessing the fact that $D$ must be large. The estimates in Theorem \ref{theo:BGN} yield  $\cdist{\ell_p^k}(\ell_q^k)=\Omega\Big(k^{\frac1p-\frac1q}\Big)$ for all $2\le p<q$. Therefore, the right order of magnitude is captured in the range $2\le p<q$ but in the (very) restricted case of a finite-dimensional $\ell_p$ target that has the same dimension as the source space. In the range $1<p<2<q$ one gets $\cdist{\ell_p^k}(\ell_q^k)=\Omega\Big(k^{1-\frac1p-\frac1q}\Big)$ which does not capture the right order of magnitude and is clearly suboptimal. 
	
	\subsection{Quantitative embeddability of $\ell_q^k$ into non-positively curved spaces} 
	\label{sec:npc}
	
	Another advantage of the proof via self-improvement is that it can be extended, with a little bit more care, to cover maps taking values into non-positively curved spaces, and more generally to the context of rounded ball metric spaces.
	
	Recall that the $\eta$-approximate midpoint set of  $x,y\in \metXd$ is defined as
	\begin{align*}
		\text{Mid}(x,y,\eta) & \eqd  \left\{z\in \metX\colon \max\big\{\dX(x,z),\dX(y,z)\big\}\le \frac{1+\eta}{2}\dX(x,y)\right\}\\
		& =   B_\metX\left(x,\frac{1+\eta}{2}\dX(x,y)\right)\cap B_\metX\left(y,\frac{1+\eta}{2}\dX(x,y)\right)
	\end{align*}
	
	As usual for an arbitrary set $A\subset \metX$, $\diam(A)\eqd\sup\big\{\dX(x,y)\colon x,y\in A\big\}.$ The following definition is due to T. J. Laakso \cite{Laakso02}.
	
	A metric space $\metXd$ is a \emph{rounded ball space} if for all $t>0$ there exists $\eta(t)>0$ such that for all $x,y\in \metX$ 
	\begin{equation}\label{eq:rbinequality}
		\diam\big(\mathrm{Mid}(x,y,\eta(t))\big)<t\cdot \dX(x,y).
	\end{equation}
	
	\begin{rema}\label{rema:rb_rem}
		Note that for all $x,y\in \metX$ and $\eta>0$, $\diam(\mathrm{Mid}(x,y,\eta))\le (1+\eta)\dX(x,y)$ always holds. Therefore the rounded ball property is non-trivial only for $t\in(0,1]$ and in this case $\eta\in(0,1)$ necessarily. 
	\end{rema}
	Note that a Banach space is a rounded ball space if and only if it is uniformly convex \cite[Lemma 5.2]{Laakso02}. We can define a \emph{rounded ball modulus} $\eta_\metX$ as follows
	\begin{equation}
		\eta_\metX(t)\eqd \sup\{\eta(t)\colon \eqref{eq:rbinequality} \text{ holds for all } x,y \in \metX\}.
	\end{equation}
	We will say that $\metXd$ is a rounded space with power type $p$ if $\eta_\metX(t)\ge ct^p$. 
	
	The following contraction lemma is an extension, to the purely metric context of rounded ball spaces, of the contraction phenomenon in Lemma \ref{lemm:contraction_uc}.
	
	\begin{lemm}\label{lemm:contraction_rb}
		Let $\metYd$ be a metric space and $\vep>0$ such that $(1+\vep^q)^{1/q}\le 2$. Assume that $\cL_k(\vep,q)$ is a $(\vep,q)$-thin $k$-Laakso substructure in $\metXd$ and that $f\colon \metX \to \metY$ satisfies
		\begin{equation}\label{eq:bilip}
			\frac1A \dX(x,y)\le \dY(f(x),f(y))\le B \dX(x,y),
		\end{equation}
		for some constants $A,B>0$. Then for every $1\le \ell \le k$, if $\{s,a,m_1,m_2,b, t\}\subset \sL_k(\vep,q)$ is indexed by a copy of one of the Laakso graphs $\sL_1$ created at step $\ell$ we have:
		\begin{equation}
			\dY(f(s),f(t))\le B \dX(s,t)(1+\vep^q)^{1/q}\Big(1-\frac{1}{2}\eta_\metY(\vep/2AB)\Big).
		\end{equation} 
	\end{lemm}
	
	\begin{proof}
		Let $r>0$ be the smallest radius such that $B_{\met Y}(f(s),r)\cap B_{\met Y}(f(t),r)\supseteq\{f(m_1),f(m_2)\}$.
		Then $$r\le \max\{d_{\met Y}(f(s),f(m_1)), d_{\met Y}(f(s),f(m_2)), d_{\met Y}(f(t),f(m_1)), d_{\met Y}(f(t),f(m_2))\}$$ and it follows from \eqref{eq:bilip} and $(c_4)$ that $r\le \frac B2(1+\vep^q)^{1/q}\dX(s,t)$. On the other hand,
		\begin{equation}
			\diam \lrpar{B_{\met Y}(f(s),r)\cap B_{\met Y}(f(t),r)}\ge d_{\met Y}(f(m_1),f(m_2)),
		\end{equation}
		and thus
		\begin{align*}
			\diam \lrpar{B_{\met Y}(f(s),r)\cap B_{\met Y}(f(t),r)} & \ge  \frac{1}{A}\dX(m_1,m_2) \stackrel{(c_5)}{=}  \frac{1}{A}\vep \dX(s,t)\\
			& \stackrel{(c_4)}{=}  \frac{1}{A}\vep \Big(\frac{\dX(s,m_1)}{(1+\vep^q)^{1/q}}+ \frac{\dX(t,m_1)}{(1+\vep^q)^{1/q}}\Big)\\
			& \ge  \frac{\vep}{AB(1+\vep^q)^{1/q}}\Big(\dY(f(s),f(m_1))+\dY(f(t),f(m_1))\Big)\\
			& \ge  \frac{\vep}{2AB}\dY(f(s),f(t)),
		\end{align*}
		where in the last inequality we used our assumption on $\vep$ and the triangle inequality.
		Therefore, $r\ge \frac{1+\eta_\metY(\vep/(2AB))}{2}\dY(f(s),f(t))$ by definition of the rounded ball modulus, and 
		\begin{align*}
			d_{\met Y}(f(s),f(t)) & \le  \frac{2r}{1+\eta_\metY(\vep/2AB)}\le \frac{B(1+\vep^q)^{1/q}\dX(s,t)}{1+\eta_\metY(\vep/2AB)}\\
			& \le  B\dX(s,t)(1+\vep^q)^{1/q}\Big(1-\frac{1}{2}\eta_\metY(\vep/2AB)\Big).
		\end{align*}
		where in the last inequality we used Remark \ref{rema:rb_rem}.
	\end{proof}
	
	A slightly different implementation of the self-improvement argument gives the following extension of Theorem \ref{theo:main1} to metric spaces with rounded ball modulus with power type. We only emphasize the few points in the proof that are different.
	
	\begin{theo}\label{theo:main2}
		Let $1<p<q$ and let $\metYd$ be a rounded ball metric space with power type $p$. For all $D\ge 1$ there exists $\vep:=\vep(D,p,q,\metY)>0$ such that if $\metXd$ admits a $(\vep,q)$-thin $k$-Laakso substructure and embeds bi-Lipschitzly with distortion at most $D\ge 1$ into $\metY$  then $D=\Omega(k^{1/p-1/q})$.
	\end{theo}
	
	\begin{proof}
		Assume that for all $x,y\in \metX$ 
		\begin{equation}\label{eq:Lip2}
			\frac 1A\dX(x,y)\le \dY(f(x),f(y))\le B\dX(x,y),
		\end{equation} 
		
		with $AB\le D$.
		
		This time we define $B_j$ to be the smallest constant such that 
		\begin{equation}
			\norm{f(x)-f(y)}\le B_j \dX(x,y),
		\end{equation} 
		for all $4\times 6^{k-j}$ pairs of points $\{x,y\}$ in $\cL_k(\vep,q)$ that are indexed by vertices of a copy of $\sL_j$ in $\sL_k$ of the form $\{s,m_i\}$ or $\{m_i,t\}$ for some $i\in\{1,2\}$, where $s$ and $t$ are the farther apart vertices in $\sL_j$ whose two distinct midpoints are $m_1$ and $m_2$. 
		
		It is clear that for all $j\in\{1,\dots,k\}$, the inequalities $1\le B_j\le B$ hold. Since in the proof of Lemma \ref{lemm:contraction_rb} we have only used the upper bound in $\eqref{eq:Lip2}$ for pairs of points of the form described in the definition of $B_{j-1}$,  proceeding as in the proof of Theorem \ref{theo:main1} we show that 
		
		
		
		
		\begin{equation*}
			AB_{j-1}-AB_{j}  \ge  AB\left(\frac{c\vep^p}{2^{p+1}(AB)^p}-(2\vep)^q\right).
		\end{equation*}
		If we let $\vep=\gamma (AB)^{-\frac{p}{q-p}}$ for some small enough $\gamma$ to be chosen later (and that depends only on $p,q$, and $c$), then 
		\begin{equation*}
			AB_{j-1}-AB_{j} \ge (AB)^{1-\frac{pq}{q-p}}(\frac{c}{2^{p+1}}\gamma^p-2^q\gamma^q).
		\end{equation*}
		If we choose $0<\gamma<\Big(\frac{c}{2^{p+q+2}}\Big)^{1/(q-p)}$ we have $\frac{c}{2^{p+1}}\gamma^p-2^q\gamma^q\ge \frac{c}{2^{p+2}}\gamma^p>0$. Hence $AB_{j-1}-AB_{j} \ge\frac{c}{2^{p+2}}\gamma^p (AB)^{1-\frac{pq}{q-p}}$ and summing over $j=2,\dots, k$ we conclude that 
		$AB\gtrsim k^{1/p-1/q}$.
	\end{proof}

	We now identify a $4$-point inequality that implies the rounded ball property with power type $p$.
	
	\begin{lemm}
		Let $\metXd$ be a metric space and $p\in(0,\infty)$. If there exists $C\in(0,2^p]$ such that for all $x_1,x_2,x_3,x_4\in \metX$ we have  
		\begin{equation}\label{eq:4-point}
			\dX(x_1,x_3)^p+\dX(x_2,x_4)^p\le \frac{C}{4}\Big(\dX(x_1,x_2)^p+\dX(x_2,x_3)^p+\dX(x_3,x_4)^p+\dX(x_4,x_1)^p\Big),
		\end{equation}
		
		then $\metX$ is a rounded ball space with $\eta_\metX(t)\ge t^p/(2^p-1)$ if $p\ge 1$ and with $\eta_\metX(t)\ge t$ if $p\in(0,1)$.
	\end{lemm}
	
	\begin{proof}
		Fix $t>0$ and let $x,y\in \metX$ and $\eta \in (0,1)$. If $\mathrm{Mid}(x,y,\eta)$ is empty or reduced to a single point there is nothing to prove. Otherwise, let $w\neq z\in \mathrm{Mid}(x,y,\eta)$. It follows from \eqref{eq:4-point} that
		
		$$\dX(x,y)^p+\dX(w,z)^p\le \frac{C}{4}(\dX(x,w)^p+\dX(w,y)^p+\dX(y,z)^p+\dX(z,x)^p),$$
		and by the definition of $\mathrm{Mid}(x,y,\eta)$, we have
		
		\begin{equation*}
			\dX(w,z)^p\le \Big(C\frac{(1+\eta)^p}{2^p}-1\Big)\dX(x,y)^p.
		\end{equation*}
		
		And since $C\le 2^p$,
		\begin{equation*}
			\dX(w,z)\le ((1+\eta)^p-1)^{\frac{1}{p}}\dX(x,y).
		\end{equation*}
		If $p\ge 1$ then $((1+\eta)^p-1)^{\frac{1}{p}}\le (2^p-1)^{1/p}\eta^{1/p}$, and if $\eta=t^p/(2^p-1)$, then
		$$\diam(\mathrm{Mid}(x,y,\eta))<t \dX(x,y).$$
		If $p\in(0,1)$ then $((1+\eta)^p-1)^{\frac{1}{p}}\le \eta$, and $\eta=t$ implies that
		$$\diam(\mathrm{Mid}(x,y,\eta))<t \dX(x,y).\qedhere$$
	\end{proof}
	
	Inequality \eqref{eq:4-point} when $p=2$ and $C=4$ is well known under various names: quadrilateral inequality, roundness $2$, Enflo type $2$ with constant $1$. It was proved by Berg and Nikolaev \cite{BergNikolaev07} (see also \cite{BergNikolaev08} or \cite{Sato09}) that the quadrilateral inequality characterizes CAT(0)-spaces amongst geodesic metric spaces and that CAT(0)-spaces coincide with Alexandrov spaces of non-positive curvatures; and this provides a rather large class of metric spaces which are rounded ball with power type $2$. It is not difficult to show that ultrametric spaces satisfy inequality \eqref{eq:4-point} with $p=1$ and $C=2$. We give one example of an application of Theorem \ref{theo:main2}.
	
	\begin{coro}\label{coro:round} 
		If $q>2$ and $\metYd$ is a metric space with roundness~2, in particular an Alexandrov space of non-positive curvature, then
		\begin{equation*}
			\cdist{\metY}(\ell_q^k)=\Omega\Big(k^{\frac12-\frac1q}\Big).
		\end{equation*}
	\end{coro}
	
	\begin{rema}
		To the best of our knowledge, the only known proof of Corollary \ref{coro:round} can be found in the work of Eskenazis, Mendel, and Naor in \cite{EMN19} where it was shown that Alexandrov spaces of non-positive curvature have metric cotype $2$. This is a particular case of a much deeper result which says that $q$-barycentric metric spaces have \emph{sharp} metric cotype $q$, and whose proof partly relies on a version of Pisier's martingale inequality in the context of nonlinear martingales.
	\end{rema}

	\section{Embeddability obstruction via thin $\aleph_0$-branching diamond substructures}
	\label{sec:auc}
	
	Using the self-improvement argument together with the smallness of approximate midpoint sets to prove Theorem \ref{theo:dimreduction} has the other advantage of being easily generalizable to the asymptotic setting. It is well-known that the size of a $t$-approximate metric midpoint set in an asymptotically uniformly convex Banach spaces is ``small''. By ``small'' we mean that the set is included in the (Banach space) sum of a compact set and a ball of small radius. Therefore the techniques from the previous sections can be adequately modified to show that the presence of countably branching versions of the Laakso-type substructure is a bi-Lipschitz embeddability obstruction. A similar fact for countably branching diamond and Laakso graphs was first proved in \cite{Baudier-et-all-JFA} and generalized in \cite{Swift18}.
	
	The only reason to work with Laakso-type substructures in the previous sections was to produce spaces with the doubling property. In the asymptotic setting, we need to work with substructures whose underlying graphs have vertices with countably many neighbors and fail the doubling property altogether. Therefore, we will only consider simpler diamond-type substructures. 
 
	As noted in \cite{Baudier-et-all-JFA} it is more convenient to work with the notion of asymptotic midpoint uniform convexity. Let $\banX$ be a Banach space and $t\in(0,1)$. Define
	\[\tilde{\delta}_\banX(t)\eqd\inf_{x\in S_\banX}\sup_{Z\in \mathrm{cof}(\banX)}\inf_{z\in S_Z}\max\{\|x+tz\|, \|x-tz\|\}-1.\]
	The norm of $\banX$ is said to be asymptotically midpoint uniformly convex if $\tilde{\delta}_\banX(t)>0$ for every $t\in(0,1)$. Being asymptotically midpoint uniformly convexifiable is formally weaker than being asymptotically uniformly convexifiable. However, it is still open whether asymptotic uniform convexity and asymptotic midpoint uniform convexity are equivalent notions up to renorming. We now recall some facts that we will need which can be found in \cite{Baudier-et-all-JFA}. A characterization of asymptotic midpoint uniformly convex norms was given in \cite{DKRRZ13} in terms of the Kuratowski measure of noncompactness of approximate midpoint sets. Recall that the Kuratowski measure of noncompactness of a subset $S$ of a metric space, denoted by $\alpha(S)$, is defined as the infimum of all $\vep>0$ such that $S$ can be covered by a finite number of sets of diameter less than $\vep$. Note that it is a property of the metric.

	In \cite{DKRRZ13} it was shown that a Banach space $\banX$ is asymptotically midpoint uniformly convex if and only if $$\ds\lim_{t\rightarrow0}\sup_{x\in S_\banX}\alpha(\mathrm{Mid}(-x,x,t))=0.$$ To prove the main result of this section we need the following lemma which is a particular case of Lemma 4.3~ in \cite{Baudier-et-all-JFA}.
	
	\begin{lemm}\label{L:5.3}If the norm of a Banach space $\banX$ is asymptotically midpoint uniformly convex, then for every $t\in(0,1)$ and every $x,y\in \banX$ there exists a finite subset $S$ of $\banX$ such that
		\begin{equation}
			\mathrm{Mid}(x,y,\tilde{\delta}_\banX(t)/4)\subset S+2t\norm{x-y}B_\banX.
		\end{equation}
	\end{lemm}
	
	We define thin diamond substructures that can be used to prove non-embeddability results.
	
	\begin{defi}[Thin $\kappa$-branching diamond substructure]
		Let $p\in[1,\infty)$, $\vep>0$, $\kappa$ be a cardinal number, and $I$ a set of cardinality $\kappa$. For $k\in \bN$, we say that a metric space $\metX$ admits a \emph{$(\vep,p)$-thin $\kappa$-branching $k$-diamond substructure} if there exists a collection $\cD^\kappa_k(\vep,p)$ of points indexed by $\sD^\kappa_k$ such that for every $1\le \ell \le k$ if $\{s,\{m_i\}_{i\in I}, t\}\subset \cD^\kappa_k$ is indexed by a copy of one of the diamond created at step $\ell$ then:
		\begin{enumerate}
			\item[($d_1$)]\label{eq:d1} $\dX(s,m_i)=\dX(m_i,t)=\frac12 (1+(2\vep)^p)^{1/p}\dX(s,t)$, \quad for all $i\in I$
			\item[($d_2$)]\label{eq:d2} $\dX(m_i,m_j)=2^{1-1/p}\vep\cdot \dX(s,t)$ for all $i\neq j$.
		\end{enumerate}
	\end{defi}
	
	In Lemma \ref{lemm:thindiamond} below, we provide a construction of a $(\vep,p)$-thin $\aleph_0$-branching $k$-diamond substructure in $L_p$-spaces, which in turns implies for all $p\in[1,\infty)$, $k\in\bN$, and $\vep>0$ the existence of an $(\vep,p)$-thin $\aleph_0$-branching $k$-diamond substructure.
	
	\begin{lemm}\label{lemm:thindiamond}
		For every $p\in[1,\infty)$, every $\vep>0$, and every $k\in\bN$; $L_p$ admits a $(\vep,p)$-thin $\aleph_0$-branching $k$-diamond substructure.
	\end{lemm}
	\begin{proof}
		Let $\chi_{i,j,k}$ stand in for the characteristic function $\chi_{\left[k+\frac{i-1}{2^j},k+\frac{i}{2^j}\right]}$.
		Fix $\varepsilon>0$.  The $(\vep,p)$-thin $\aleph_0$-branching $k$-diamond substructure in $L_p$ with parameter $\varepsilon>0$ is defined recursively as follows. For simplicity, we start the induction with the $0$-diamond graph $\sD^\omega_0$ to be a single edge with endpoint $s$ and $t$, and (again identifying the points in $\cD^\omega_k(\vep,p)$ with the vertices of $\sD^\omega_k$) we define $\cD_0^\omega(\vep,p):= \{s,t\}$ by $s\eqd \chi_{[0,1]}$ and $t \eqd -\chi_{[0,1]}$ and the conditions are vacuously satisfied.
		Suppose now that $\cD_k^\omega$ has already been defined such that $\cD_k^\omega\subseteq L_p[0,k+1]$. To construct $\cD_{k+1}^\omega$ we introduce for every edge $\{s,t\}\in \cD_{k}^\omega$ and $i\in \bN$ a ``midpoint'' as follows: 
		\begin{equation}
			m_i=\frac{s+t}{2}+\sum_{r=1}^{2^i}(-1)^r\varepsilon \|s-t\|_p\chi_{r,i,k+1}.
		\end{equation}
		
		Then, 
		
		\begin{align*}
			\norm{s-m_i}_p^p & = \bnorm{\frac{s-t}{2}-\sum_{r=1}^{2^i}(-1)^r\varepsilon\|s-t\|_p\chi_{r,i,k+1}}^p_p = \bnorm{\frac{s-t}{2}}_p^p+\bnorm{\sum_{r=1}^{2^i}(-1)^r\varepsilon \|s-t\|_p\chi_{r,i,k+1}}^p_p\\
			& = \bnorm{\frac{s-t}{2}}_p^p+\varepsilon^p\norm{s-t}_p^p = \frac{(1+(2\varepsilon)^p)}{2^p}\norm{s-t}_p^p, 
		\end{align*}
		
		wherein the second equality we used the fact that the vectors have disjoint supports (in $[0,k+1]$ and $[k+1,k+2]$, respectively).
		
		For $i<j$, observe that $\chi_{r,i,k+1}=\sum_{\ell=(r-1)2^{j-i}+1}^{r2^{j-i}}\chi_{\ell,j,k+1}$, and so
		\begin{align*}
			\norm{m_i-m_j}_p^p & = \bnorm{\sum_{r=1}^{2^{i}}(-1)^r \vep \norm{s-t}_p \chi_{r,i,k+1}-\sum_{r=1}^{2^{j}}(-1)^r \vep \norm{s-t}_p \chi_{r,j,k+1}}^p_p\\
			& =  \vep^p \norm{s-t}^p_p\left\|\sum_{r=1}^{2^{i}}\sum_{\ell=(r-1)2^{j-i}+1}^{r2^{j-i}}\Big((-1)^{r}-(-1)^{\ell}\Big)\chi_{\ell,j,k+1}\right\|_p^p\\
			& =  \vep^p \norm{s-t}^p_p \left(\sum_{r=1}^{2^{i}}\sum_{\ell=(r-1)2^{j-i}+1}^{r2^{j-i}}\int_{k+1+\frac{\ell-1}{2^{j}}}^{k+1+\frac{\ell}{2^{j}}}\left|(-1)^{r}-(-1)^{\ell}\right|^p\mathop{dx}\right)\\
			& =  \vep^p \norm{s-t}^p_p \left(\sum_{r=1}^{2^{i}} \frac{2^{j-i}}{2} \cdot 2^{-j} \cdot 2^p\right)\\
			& =  \vep^p \norm{s-t}^p_p \cdot\frac{1}{2}\cdot 2^p\\
			& =  2^{p-1} \vep^p \norm{s-t}^p_p\qedhere\\
		\end{align*}

	\end{proof}
	
	Next, we prove the contraction principle that is needed in the asymptotic setting.
	
	\begin{lemm}\label{lemm:contraction_amuc}
		Let $\vep>0$ such that $(1+(2\vep)^p)^{1/p}\le 2$ and let $\kappa$ be an infinite cardinality. Assume that $\cD^\kappa_k (\vep,p)$ is a $(\vep,p)$-thin $\kappa$-branching $k$-diamond substructure in $\metXd$ and that $f\colon \metX \to \banYn$ is a bi-Lipschitz embedding with distortion $D$. Then for every $1\le \ell \le k$, if $\{s,\{m_i\}_{i\in I}, t\}\subset \cD^\kappa_k$ is indexed by a copy of one of the diamond graph $\sD_1^\kappa$ created at step $\ell$, we have: 
		\begin{equation}
			\|f(s)-f(t)\|\le D\dX(s,t)(1+(2\vep)^p)^{1/p}\left(1-\frac15\tilde{\delta}_\banY\left(\frac{\vep}{16D}\right)\right).
		\end{equation} 
	\end{lemm}
	
	\begin{proof}
		Assume that for all $x,y\in \metX$ 
		\begin{equation}\label{eq:Lip5}
			\dX(x,y)\le \norm{f(x)-f(y)}\le D\dX(x,y).
		\end{equation} 
		
		We claim that there exists $j\in\mathbb{N}$ such that 
		\begin{equation}\label{eq:mid}
			f(m_j)\notin\text{Mid}\left(f(s),f(t),\frac{1}{4}\tilde{\delta}_\banY\left(\frac{\vep}{16D}\right)\right).
		\end{equation}

		Assuming for a moment that \eqref{eq:mid} holds, then we have either
		\begin{equation*}
			\|f(m_j)-f(t)\|>\frac12 \left(1+\frac{1}{4}\tilde{\delta}_\banY\left(\frac{\vep}{16D}\right)\right)\|f(s)-f(t)\|
		\end{equation*}
		or
		\begin{equation*}
			\|f(m_j)-f(s)\|>\frac12 \left(1+\frac{1}{4}\tilde{\delta}_\banY\left(\frac{\vep}{16D}\right)\right)\|f(s)-f(t)\|.
		\end{equation*}
		In both cases it follows from \eqref{eq:Lip5} and condition $(d_1)$ above that 
		\begin{align*}
			\|f(s)-f(t)\|& < D\dX(s,t) (1+(2\vep)^p)^{1/p}\left(1+\frac{1}{4}\tilde{\delta}_X\left(\frac{\vep}{16D}\right)\right)^{-1}\\
			& \le D\dX(s,t) (1+(2\vep)^p)^{1/p}\left(1-\frac{1}{5}\tilde{\delta}_X\left(\frac{\vep}{16D}\right)\right).
		\end{align*}

		It remains to prove \eqref{eq:mid}. By Lemma \ref{L:5.3} there exists a finite subset $S:=\{z_1,\dots, z_n\}\subset \banY$ such that 
		\begin{equation*}
			\text{Mid}\left(f(s),f(t),\frac{1}{4}\tilde{\delta}_\banY\left(\frac{\vep}{16D}\right)\right)\subset S+\frac{\vep}{8D}\norm{f(s)-f(t)}B_\banY.
		\end{equation*}
		If for every $i\in\bN$, 
		\begin{equation*}
			f(m_i)\in\text{Mid}\left(f(s),f(t),\frac{1}{4}\tilde{\delta}_\banY\left(\frac{\vep}{16D}\right)\right),
		\end{equation*}
		then $f(m_i)=z_{n_i}+y_i$ with $z_{n_i}\in S$ and $y_i\in \banY$ so that 
		\begin{equation*}
			\norm{y_i}\le\frac{\vep}{8D}\norm{f(s)-f(t)}.
		\end{equation*}
		Therefore, for all $i\neq j$, 
		
		\begin{align*}
			\norm{z_{n_i}-z_{n_j}} &\ge\norm{f(m_i)-f(m_j)}-\norm{y_i-y_j}\\
			&\ge \dX(m_i,m_j)-\frac{\vep}{4D}\norm{f(s)-f(t)}\\
			&\ge \dX(m_i,m_j)-\frac{\vep}{4D}(\norm{f(s)-f(m_i)}+\norm{f(m_i)-f(t)})\\
			&\ge 2^{1-1/p}\vep\cdot \dX(s,t) -\frac{\vep}{4}(1+(2\vep)^p)^{1/p}\dX(s,t)\\
			&\ge 2^{1-1/p}\vep\cdot \dX(s,t) -\frac{\vep}{2}\dX(s,t)\\
			&\ge \frac12\vep\cdot \dX(s,t)>0, 
		\end{align*}
		which contradicts the fact that $S$ is finite.
	\end{proof}
	
	Since in the proof of Lemma \ref{lemm:contraction_amuc} we were careful to only use the upper bound in \eqref{eq:Lip5} for pairs of points of the form $\{s,m_i\}$ or $\{t,m_i\}$, the derivation of Theorem \ref{theo:main2} below from Lemma \ref{lemm:contraction_amuc} is by now standard and thus omitted.
	
	\begin{theo}\label{theo:main2}
		Let $1\le p<q$ and assume that $\metXd$ admits a bi-Lipschitz embedding with distortion $D$ into a $p$-asymptotically midpoint uniformly convex Banach space $\banY$. There exists $\vep:=\vep(p,q,D,\banY)>0$ such that if $\metX$ admits a $(\vep,q)$-thin $\aleph_0$-branching $k$-diamond substructure then $D=\Omega(k^{1/p-1/q})$.
	\end{theo}
	
	The following consequence is immediate.
	
	\begin{coro}\label{cor:appli}
		$L_q[0,1]$ does not bi-Lipschitzly embed into any $p$-asymptotically midpoint uniformly convex Banach space if $q>p\ge 1$. In particular, $L_q[0,1]$ does not bi-Lipschitzly embed into $\ell_p$ if $q>p\ge 1$.
	\end{coro}
	
	\begin{rema}
		Corollary \ref{cor:appli} is not new since it can be shown using classical approximate midpoint techniques (see \cite[Chapter 10, Section 2]{BenyaminiLindenstrauss00} or \cite{KaltonLova08} for instance). The classical approximate midpoint technique provides an obstruction of qualitative nature and relies on some linear arguments but it can handle weaker notions of embeddings. Our proof of Theorem \ref{theo:main2}, and in turn of Corollary \ref{cor:appli}, identifies concrete and purely metric structures that provide quantitative obstructions to bi-Lipschitz embeddings. 
	\end{rema}
	
	\bibliographystyle{alpha}

\end{document}